\documentclass[reqno]{amsart}

\usepackage{amssymb}
\usepackage{amsthm}
\usepackage{amsmath}
\usepackage{amsfonts}
\usepackage{comment}
\usepackage{enumerate, enumitem}
\usepackage{graphicx}

\newtheorem{thm}{Theorem}[section]

\newtheorem{lem}[thm]{Lemma}
\newtheorem{prop}[thm]{Proposition}
\newtheorem{cor}[thm]{Corollary}
\newtheorem{theoremintro}{Theorem}
\newtheorem{corollaryintro}{Corollary}

\newcommand{\bbold}{\mathbb}

\newcommand{\nn}{\mathbb{N}}

\newcommand{\zz}{\mathbb{Z}}
\def\R { {\bbold R} }
\def\Z { {\bbold Z} }
\def\T { {\bbold T} }

\def \d{\operatorname{d}}
\def \<{\langle}
\def \>{\rangle}

\def \k {{{\boldsymbol{k}}}}
\def \Th{\operatorname{Th}}
\def\Mo{\operatorname{Mo}}

\DeclareFontFamily{OMS}{smallo}{}
\DeclareFontShape{OMS}{smallo}{m}{n}{<->s*[.65]cmsy10}{}
\DeclareSymbolFont{smallo@m}{OMS}{smallo}{m}{n}
\DeclareMathSymbol{\smallo}{\mathord}{smallo@m}{79}

\DeclareFontFamily{U}{fsy}{}
\DeclareFontShape{U}{fsy}{m}{n}{<->s*[.9]psyr}{}
\DeclareSymbolFont{der@m}{U}{fsy}{m}{n}
\DeclareMathSymbol{\der}{\mathord}{der@m}{182}

\begin{document}

\title[\resizebox{4.7in}{!}{An Ax-Kochen-Ershov Theorem for Monotone Differential-Henselian Fields}]{An Ax-Kochen-Ershov Theorem for
Monotone Differential-Henselian Fields}

\author[Hakobyan]{Tigran Hakobyan}
\address{Department of Mathematics\\
University of Illinois at Urbana-Cham\-paign\\
Urbana, IL 61801\\
U.S.A.}
\email{hakobya2@illinois.edu}

\begin{abstract} Scanlon~\cite{S} proves Ax-Kochen-Ershov type results for differential-henselian monotone valued differential
fields with many constants. We show how to get rid of
the condition {\em with many constants}. 
\end{abstract}

\maketitle
\section*{Introduction}

\noindent
Let $\k$ be a differential field (always of characteristic $0$ in this paper,
with a single distinguished derivation). Let also an ordered abelian group $\Gamma$ be given. This gives rise to the Hahn field 
$K=\k((t^\Gamma))$, to be considered in the usual way as a valued
field. We extend the derivation 
$\der$ of $\k$ to a derivation on $K$ by
$$\der(\sum_{\gamma} a_{\gamma}t^\gamma)\ :=\ \sum_{\gamma} \der(a_{\gamma})t^\gamma.$$ 
Scanlon~\cite{S} extends the Ax-Kochen-Ershov theorem (see \cite{AK}, \cite{E}) to 
this differential setting. This includes requiring that $\k$ is linearly surjective in the sense that for each nonzero linear differential
operator $A=a_0+a_1\der + \dots +a_n\der^n$ over $\k$ we have
$A(\k)=\k$. Under this assumption, $K$ is differential-henselian
(see Section~\ref{sec:pre} for this notion), and the theory $\Th(K)$ of $K$ as a valued differential field (see also Section~\ref{sec:pre} for this) is completely axiomatized by: \begin{enumerate}
\item the axiom that there are many constants;
\item the theory $\Th(\k)$ of the differential residue field $\k$;
\item the theory $\Th(\Gamma)$ of the ordered abelian value group;
\item the axioms for differential-henselian valued fields.
\end{enumerate}
As to (1), having many constants means that every element of the differential field has the same valuation as some element of its constant field. This holds for $K$ as above (whether or not $\k$ is linearly surjective) because the constant field
of $K$ is $C_K=C_{\k}((t^{\Gamma}))$. This axiom plays an important role in some proofs of \cite{S}. Below we drop the ``many constants'' axiom
and generalize the theorem above to a much
larger class of differential-henselian valued fields. This 
involves a more general way of extending the derivation of $\k$ to $K$.

In more detail, let $c: \Gamma \to \k$ be an additive map. Then
the derivation $\der$ of $\k$ extends to a derivation $\der_c$ of $K$ by setting
$$ \der_c(\sum_{\gamma} a_{\gamma}t^{\gamma})\ :=\ \sum_\gamma \big(\der(a_\gamma) + c(\gamma) a_\gamma \big)t^\gamma.$$
Thus $\der_c$ is the unique derivation on $K$ that extends $\der$, respects infinite sums, and satisfies $\der_c(t^\gamma)=c(\gamma)t^\gamma$ for all $\gamma$. 
The earlier case has $c(\gamma)=0$ for all $\gamma$. Another case is where
$\k$ contains $\R$ as a subfield, $\Gamma=\R$, and $c: \R\to \k$ is the 
inclusion map; then $\der_c(t^r)=rt^r$ for $r\in \R$.

Let $K_c$ be the valued differential field $K$ with $\der_c$ as its distinguished derivation. 
Assume in addition that $\k$ is linearly surjective. Then $K_c$ is differential-henselian, and Scanlon's theorem above generalizes as follows:   

\begin{theoremintro}\label{F1} The theory $\Th(K_c)$ is completely determined by
$\Th(\k,\Gamma;c)$, where $(\k, \Gamma;c)$ is the $2$-sorted structure
consisting of the differential field $\k$, the ordered abelian group 
$\Gamma$, and the additive map $c: \Gamma \to \k$.
\end{theoremintro}

\noindent
We actually prove in Section 2 a stronger version with the one-sorted structure $K_c$ 
expanded to a $2$-sorted one, with $\Gamma$ as the underlying set for the second sort, and as extra primitives the cross-section $\gamma \mapsto t^\gamma: \Gamma\to K$, the set $\k\subseteq K$, and the map
$c: \Gamma \to \k$. 

The question arises: which complete theories of valued differential fields
are covered by Theorem~\ref{F1}?
 The answer involves the notion of monotonicity:
a valued differential field $F$ with valuation $v$ is said to be 
{\em monotone\/} if $v(f')\ge v(f)$ for all $f\in F$; as usual, $f'$ denotes the derivative of $f\in F$ with respect to the distinguished derivation of $F$. 
The valued differential fields $K_c$ are all clearly monotone. 
We show:

\begin{theoremintro}\label{F2} Every monotone differential-henselian valued field is elementarily equivalent to some $K_c$ as in Theorem~\ref{F1}. 
\end{theoremintro}

\noindent
This is proved in Section 3 and is analogous to the result from \cite{S} that any differential-henselian valued field with many constants is elementarily equivalent to some $K$ as in
Scanlon's theorem stated in the beginning of this Introduction. 
(In fact, that result follows from the ``complete axiomatization''
given in that theorem.)

Theorem~\ref{F2} has a nice algebraic consequence, generalizing \cite[Corollary 8.0.2]{ADAMTT}:

\begin{corollaryintro} \label{F3}
If a valued differential field $F$ is monotone and differential-henselian, then every valued differential field extension of $F$ that is algebraic over $F$ is also (monotone and) differential-henselian.
\end{corollaryintro}

\noindent
See Section 4. To state further results it is convenient to introduce some
notation. Let $F$ be a differential field. For nonzero
$f\in F$ we set $f^\dagger := f' / f$ and $F^\dagger~ :=~ \{f^\dagger~ :~f~ \in~F^{\times}\}$, where $F^{\times} := F \setminus\{0\}.$

So far our only assumption on $c: \Gamma \to \k$ is that it is additive, but
the case $c(\Gamma)\cap \k^\dagger = \{0\}$ is of particular interest: it is not
hard to show that then the constant field of $K_c$ is $C_\k ((t^\Delta))$,
where the value group $\Delta$ of the constant field equals $\ker(c)$
 and is a pure subgroup of
$\Gamma$. Conversely (see Section 3):

\begin{theoremintro} \label{F4}
Every monotone differential-henselian valued field $F$
such that $v(C_F^\times)$ is pure in $v(F^\times)$ is elementarily equivalent
to some $K_c$ as in Theorem \ref{F1} with $c(\Gamma)\cap~\k^\dagger = \{0\}$.
\end{theoremintro}

\noindent
The referee showed us an example of a monotone henselian valued differential field $F$ for which $v(C_F^\times)$ is not pure in $v(F^\times)$.
In Section 4 we give an example of a monotone differential-henselian 
field $F$ such that $v(C_F^\times)$ is not pure in $v(F^\times)$.

The hypothesis
of Theorem~\ref{F4} 
that $v(C_F^\times)$ is pure in $v(F^\times)$ holds if the residue field is
algebraically closed or real closed (see Section 4). It includes also
the case of main interest to us, where $F$ has few constants, that is, the valuation is trivial on $C_F$.
In that case any $c$ as in Theorem~\ref{F4} is injective by 
Corollary~\ref{fewConstants}. 

Section 3 contains examples of
additive maps $c: \Gamma\to \k$ for which $K_c$ has few constants, including a case where $\Th(K_c)$ is decidable. Two of those examples show that in Theorem~\ref{F1}, even when we have few constants, the traditional Ax-Kochen-Ershov principle without the map $c$ does not hold. (It does hold in Scanlon's theorem where
$c=0$, but in general we do not expect to have a $c$ that is definable in the 
valued differential field structure.)

\section{Preliminaries}\label{sec:pre}

\medskip\noindent
%
Adopting terminology from \cite{ADAMTT},
a {\em valued differential field\/} is
a differential field $K$ together with a
(Krull) valuation $v: K^{\times} \to \Gamma$ whose
residue field $\mathbf{k}$ 
$:=\mathcal{O}/\smallo$ has characteristic 
zero; here $\Gamma=v(K^\times)$ is 
the value group, and we also let $\mathcal{O}=\mathcal{O}_K$ denote the valuation ring 
of $v$ with maximal ideal $\smallo$, and let
$$C\ =\ C_K\ :=\ \{f\in K:\ f'=0\}$$ 
denote the constant field of the differential field $K$. 
We use notation from~\cite{ADAMTT}: for elements $a,b$ of a 
valued field with valuation $v$ we set 
$$a\asymp b:\Leftrightarrow va=vb, \quad a\preceq b\Leftrightarrow b\succeq a:\Leftrightarrow va\ge vb, \quad a\prec b \Leftrightarrow b\succ a:\Leftrightarrow va > vb.$$

Let $K$ be a valued differential field as above, and let $\der$ 
be its derivation. 
We say that
$K$ has {\em many constants\/} if $v(C^\times)=\Gamma$. 
We say that the derivation of $K$ is {\em small\/} if $\der(\smallo)\subseteq \smallo$.
If $K$, with a small derivation, has many constants, then $K$ is {\em monotone\/} in the sense of 
\cite{C}, that is, $v(f) \le v(f')$ for all $f\in K$. We say that
$K$ has {\em few constants\/} if $v(C^\times)=\{0\}$.
Note: if $K$ is monotone, then its derivation
is small; if the derivation of $K$ is small, then $\der$ is continuous
with respect to the valuation topology on $K$. Note also that if 
$K$ is monotone, then so is any valued differential field extension 
with small derivation and the same value group as $K$.

{\em From now on we assume that the derivation
of $K$ is small}. This has the effect (see \cite{C} or
\cite[Lemma 4.4.2]{ADAMTT}) that also 
$\der(\mathcal{O})\subseteq \mathcal{O}$, and so $\der$ induces a derivation
on the residue field; we view $\mathbf{k}$ below as
equipped with this induced derivation, and refer to it as the 
{\em differential residue field of $K$}.

\medskip\noindent
We say that $K$ is {\em differential-henselian\/} (for short: 
{\em $\operatorname{d}$-henselian}) if every differential polynomial 
$P\in \mathcal{O}\{Y\}=\mathcal{O}[Y, Y', Y'',\dots]$ whose reduction
$\overline{P}\in \mathbf{k}\{Y\}$ has total degree $1$ has a zero in 
$\mathcal{O}$. (Note that for ordinary 
polynomials $P\in \mathcal{O}[Y]$ this requirement
defines the usual notion of a henselian valued field, that is, 
a valued field whose valuation ring is henselian as a local ring.) 

If $K$ is $\operatorname{d}$-henselian, then its differential residue field
is clearly {\em linearly surjective}: any linear differential equation
$y^{(n)}+ a_{n-1}y^{(n-1)} + \cdots + a_0y = b$ with coefficients $a_i,b\in  \mathbf{k}$ has a solution in $\mathbf{k}$. This is a key
constraint on our notion of $\operatorname{d}$-henselianity. 
If $K$ is $\operatorname{d}$-henselian, then $\mathbf{k}$ has a {\em lift to $K$}, 
meaning, a differential subfield of $K$
contained in $\mathcal{O}$ that maps isomorphically 
onto $\mathbf{k}$ under the canonical map from $\mathcal{O}$ onto $\mathbf{k}$; see
\cite[7.1.3]{ADAMTT}. Other items from \cite{ADAMTT} that are relevant in this paper
are the following differential analogues of 
Hensel's Lemma and of results due to
Ostrowski/Krull/Kaplansky on valued fields:
\begin{enumerate}
\item[(DV1)] If the derivation of $\mathbf{k}$ is nontrivial, then
$K$ has a spherically complete immediate valued differential field extension 
with small derivation; \cite[6.9.5]{ADAMTT}. 
\item[(DV2)] If $\mathbf{k}$ is linearly surjective and 
$K$ is spherically complete, then $K$ is $\d$-henselian;
\cite[7.0.2]{ADAMTT}.

\item[(DV3)] If $\mathbf{k}$ is linearly surjective and $K$ is monotone, 
then any two spherically complete immediate monotone
valued differential field extensions of $K$ are isomorphic over $K$; 
\cite[7.4.3]{ADAMTT}.
\end{enumerate}

\noindent
We also need a model-theoretic variant of (DV3): 

\begin{enumerate}
\item[(DV4)] Suppose $\mathbf{k}$ is linearly surjective and $K$ is monotone with $v(K^\times) \not= \{ 0 \}$.
Let $K^{\bullet}$ be a spherically complete immediate
valued differential field extension of $K$. Then  $K^{\bullet}$ can be embedded
over $K$ into any $|v(K^\times)|^+$-saturated $\d$-henselian monotone valued differential
field extension of $K$; \cite[7.4.5]{ADAMTT}.
\end{enumerate}

\section{Elementary equivalence of monotone differential-henselian fields}
\noindent
In this section we obtain Theorem $\ref{F1}$ from the introduction as a 
consequence of a more precise result in a 2-sorted setting.
We consider 2-sorted structures
$$\mathcal{K}\ =\ (K, \Gamma; v, s, c),$$
where $K$ is a differential field equipped with a differential
subfield $\k$ (singled out by a unary 
predicate symbol), $\Gamma$ is an ordered abelian group,
$v: K^\times \to \Gamma=v(K^\times)$ is a valuation that makes $K$ into
a monotone valued differential field such that $\k \subseteq K$ is a lift
of the differential residue field, $s: \Gamma \to K^\times$ is a cross-section
of $v$ (that is, $s$ is a group morphism and $v\circ s=\text{id}_{\Gamma}$),
and $c: \Gamma \to \k$ satisfies 
$c(\gamma)=s(\gamma)^\dagger$ for all $\gamma \in \Gamma$ (so $c$ is additive). We construe these $\mathcal{K}$ as $L_2$-structures
for a natural $2$-sorted language $L_2$ (with unary function symbols for
$v$, $s$, and $c$). We have an obvious
set $\Mo(\ell, s,c)$ of $L_2$-sentences whose models are exactly these $\mathcal{K}$; the ``$\ell$'' is to indicate the presence of a lift.

For example, for $K=\k((t^\Gamma))$ as in the introduction 
and additive $c: \Gamma\to \k$ we consider $K_c$ as a model of $\Mo(\ell, s,c)$
in the obvious way by taking $\k\subseteq K$ as lift, and $\gamma\mapsto t^\gamma$ as cross-section. 

\begin{thm}\label{MainTheorem} 
If $\mathcal{K}$ is $\d$-henselian, then
$\Th(\mathcal{K})$ is axiomatized by:
\begin{enumerate}[font=\normalfont]
\item $\Mo(\ell, s,c)$;
\item the axioms for $\d$-henselianity;
\item Th$(\k,\Gamma;c)$ with $\k$ as differential field and $\Gamma$ as ordered abelian group.
\end{enumerate}
\end{thm}


\noindent
We first develop the required technical material, and give the proof 
of this theorem at the end of this section. Until further notice, $\mathcal{K} = (K, \Gamma; \k, v, s, c)\models \Mo(\ell, s,c)$. 
For any subfield $E$ of $K$ we set $\Gamma_E:= v(E^\times)$.

\medskip\noindent
We define a 
{\em good subfield of $\mathcal{K}$} to be a differential subfield of $K$ such that (i) $\k \subseteq E$, (ii) $s(\Gamma_E) \subseteq E$, and (iii) $\vert{ \Gamma_E}\vert \leq \aleph_0$. Thus $\k$ is a good subfield of $\mathcal{K}$.


\begin{lem}\label{CountabilityOfValuation}
Let $E$ be a good subfield of $\mathcal{K}$ and $x\in K\setminus E$. Then 
$|\Gamma_{E(x)}|\le \aleph_0$.
\end{lem}

\noindent
This is well-known; see for example \cite[Lemma 3.1.10]{ADAMTT}.

\begin{lem}\label{GoodSubfields}
Let $E\subseteq K$ be a good subfield of $\mathcal{K}$ and $\gamma\in \Gamma\setminus \Gamma_E$, that is, $s(\gamma)\notin E$. Then $E(s(\gamma))$ is also a good subfield of $\mathcal{K}$.
\end{lem}

\begin{proof} From $c(\gamma)\in \k\subseteq E$ and $s(\gamma)' = c(\gamma)s(\gamma)$ we get that
$E(s(\gamma))$ is a differential subfield of $K$ and that
condition (i) for being a good subfield is satisfied by $E(s(\gamma))$.
For condition (ii) we distinguish two cases:

\medskip\noindent
(1) $n \gamma \in \Gamma_E$ for some $n\in \nn^{\ge 1}$. Take
$n\geq 1$ minimal with $n \gamma \in \Gamma_E$. Then $0,\gamma, 2 \gamma,\ldots, (n-1) \gamma$ are in different cosets of $\Gamma_E$, so for every $q(X)\in E[X]^{\neq}$ of degree  $<n$ we get $q(s(\gamma))\neq 0$. Hence the minimum polynomial of $s(\gamma)$ over $E$ is $X^n - s(n \gamma)$. Thus, given any $x\in E(s(\gamma))^\times$, we have 
$$x\ =\  q_0 + q_1 s(\gamma) + \ldots + q_{n-1} s(\gamma)^{n-1}$$ with $q_0,\dots, q_{n-1}\in E$, not all $0$,
so $\displaystyle v(x) = \min_{i=0,\ldots,n-1} \{ v(q_i) + i \gamma\}$. Therefore, $\Gamma_{E(s(\gamma))} = \Gamma_E + \zz \gamma$ and hence $s(\Gamma_{E(s(\gamma))})\subseteq s(\Gamma_E) \cdot s(\gamma)^{\zz}\subseteq E(s(\gamma))$.

\medskip\noindent
(2) $n \gamma\notin \Gamma_E$ for all $n\in \nn^{\ge 1}$.
Then $0,\gamma, 2\gamma, \ldots $ are in different cosets of $\Gamma_E$, so $s(\gamma)$ is transcendental over $E$ and for any polynomial $q(X)= q_0+q_1 X + \ldots + q_n X^n\in E[X]$, we have $\displaystyle v(q(s(\gamma))) = \min_{i=0,\ldots, n} \{ v(q_i) + i \gamma \}$. As in case (1) this yields 
$\Gamma_{E(s(\gamma))} = \Gamma_E + \zz \gamma$ and so
 $s(\Gamma_{E(s(\gamma))})\subseteq s(\Gamma_E) \cdot s(\gamma)^{\zz}\subseteq E(s(\gamma))$.

Thus condition (ii) of good subfields holds for $E(s(\gamma))$.
Condition (iii) is satisfied by Lemma \ref{CountabilityOfValuation}.
\end{proof}

\noindent
\textit{In the rest of this section we fix a $\d$-henselian $\mathcal{K}$.}
Let $T_{\mathcal{K}}$ be the $L_2$-theory given by (1)--(3) in Theorem~\ref{MainTheorem}. Assume CH (the Continuum Hypothesis), and let 
$$\mathcal{K}_1\ =\ (K_1, \Gamma_1; v_1, s_1, c_1), \qquad \mathcal{K}_2\ =\ (K_2, \Gamma_2; v_2, s_2, c_2)$$ be saturated models of $T_{\mathcal{K}}$ of cardinality $\aleph_1$; remarks following Corollary~\ref{cor:syntequiv} explain why we can assume CH. Then the structures $(\k_1, \Gamma_1; c_1)$ and $(\k_2, \Gamma_2; c_2)$ are also saturated of cardinality $\aleph_1$, where $\k_1$ and $\k_2$ are the lifts of the differential residue fields of $K_1$ and $K_2$ respectively. 
Since $(\k_1, \Gamma_1; c_1)$ and $(\k_2, \Gamma_2; c_2)$ are elementarily equivalent to $(\k, \Gamma; c)$, we have an isomorphism $f = (f_r, f_v)$ from 
$(\k_1, \Gamma_1; c_1)$ onto $(\k_2, \Gamma_2; c_2)$ with $f_r : \k_1 \to \k_2$ and $f_v : \Gamma_1 \to \Gamma_2$.

A map $g:E_1\to E_2$ between good subfields $E_1$ and $E_2$ of $\mathcal{K}_1$ and $\mathcal{K}_2$ respectively, will be called {\em good\/} if
\begin{enumerate}[resume*]
\item $g:E_1\to E_2$ is a differential field isomorphism,
\item $g$ extends $f_r$,
\item $f_v \circ v_1 = v_2\circ g$,
\item $g\circ s_1=s_2\circ f_v$.

\end{enumerate}

\noindent
Note that then $g$ is also an isomorphism of the valued subfield $E_1$ of $K_1$ onto the valued subfield $E_2$ of $K_2$. The map $f_r:\k_1\to \k_2$  is clearly a good map.

\begin{prop} \label{BackAndForth} $\mathcal{K}_1\cong \mathcal{K}_2$.
\end{prop}
\begin{proof}  We claim that the collection of good maps is a back-and-forth system between  $K_1$ and  $K_2$. (By the saturation assumption this yields the desired result.)
This claim holds trivially if $\Gamma_1 = \{ 0 \}$, so assume $\Gamma_1 \not= \{ 0 \}$, and thus $\Gamma_2 \not= \{ 0\}$.

\medskip\noindent
Let $g:E_1 \to E_2$ be a good map and $\gamma\in \Gamma_1\setminus \Gamma_{E_1}$. By Lemma~\ref{GoodSubfields} we have good subfields $E_1\big(s_1(\gamma)\big)$ of $\mathcal{K}_1$ and
$E_2\big(s_2(f_v(\gamma))\big)$ of $\mathcal{K}_2$.
The proof of that lemma then yields easily a good map
$$g_\gamma: E_1\big(s_1(\gamma)\big)\to E_2\big(s_2(f_v(\gamma))\big)$$
that extends $g$ with $g_\gamma\big(s_1(\gamma)\big)=s_2\big(f_v(\gamma)\big)$.

\medskip\noindent
Let $g:E_1\to E_2$ be a good map and $x\in K_1 \setminus E_1$. We show how to extend $g$ to a good map with $x$ in its domain.

By condition (i) of being a good subfield, $E_1\supseteq \k_1$ and $E_2\supseteq \k_2$. The group $\Gamma_{E_1\< x \>}$ 
is countable by Lemma \ref{CountabilityOfValuation}.
Thus by applying iteratively the construction above to elements 
$\gamma\in \Gamma_{E_1\<x\>}$, we can extend $g$ to a good map $g^1:E_1^1\to E_2^1$ with 
$\Gamma_{E_1^1}=\Gamma_{E_1 \< x \>}$. Likewise we can extend $g^1$ to a good map
$g^2:E_1^2 \to E_2^2$ with $\Gamma_{E_1^2}=\Gamma_{E_1^1 \< x \>}$.
Iterating this process and taking the union $\displaystyle E_i^{\infty} = \bigcup_n E_i^n$, for $i=1,2$, we get a good map
$g^\infty: E_1^\infty \to E_2^\infty$ extending $g$ such that $\Gamma_{E_1^{\infty}}=\Gamma_{E_1^{\infty} \< x \>}$, so the valued differential field extension $E_1^{\infty} \< x \>$ of $E_1^{\infty}$ is immediate. 
By (DV1) and (DV4) we have a spherically complete immediate valued differential field extension $E_1^{\bullet}\subseteq K_1$ of $E_1^{\infty} \< x \>$. Note that then $E_1^{\bullet}$ is also a spherically complete
immediate valued differential field extension of $E_1^{\infty}$. 
Likewise we have a spherically complete immediate valued differential field extension $E_2^{\bullet}\subseteq K_2$ of 
$E_2^{\infty}$. By (DV3) we can extend
$g^\infty$ to a valued differential field isomorphism 
$g^{\bullet}:E_1^{\bullet} \to E_2^{\bullet}$. It is clear that
then $g^{\bullet}$ is a good map extending $g$ with $x$ in its domain. 

This finishes the proof of the \textit{forth} part. The \textit{back} part is done likewise.
\end{proof}


\begin{proof}[Proof of Theorem \ref{MainTheorem}]
We can assume the Continuum Hypothesis (CH) for this argument.
(This is explained further in the remarks following Corollary~\ref{cor:syntequiv}.) Our job is to show that the theory $T_{\mathcal{K}}$ is complete. In other words, given any two models of
$T_{\mathcal{K}}$ we need to show they are elementarily equivalent. Using CH we can assume that these models are saturated of cardinality $\aleph_1$, and so they are indeed isomorphic by Proposition~\ref{BackAndForth}. \end{proof}

\noindent
Note that Theorem \ref{F1} is a consequence of 
Theorem $\ref{MainTheorem}$.

\begin{cor}\label{cor:equiv}
Suppose $\mathcal{K}_1 = (K_1, \Gamma_1; v_1, s_1, c_1)$ and $\mathcal{K}_2 = (K_2, \Gamma_2; v_2, s_2, c_2)$ are $\d$-henselian models of $\Mo(\ell,c,s)$. 
Then: $\mathcal{K}_1 \equiv \mathcal{K}_2\Longleftrightarrow (\k_1, \Gamma_1; c_1) \equiv (\k_2, \Gamma_2; c_2)$.
\end{cor}

\noindent
In connection with eliminating the use of CH we introduce the $L_2$-theory
$T$ whose models are the $\d$-henselian models of $\Mo(\ell,s,c)$. The 
structures
$(\k, \Gamma; c)$ where $\k$ is a differential field, $\Gamma$ is an ordered abelian group, and $c: \Gamma \to \k$, are  $L_c$-structures for a certain sublanguage $L_c$ of $L_2$.
Now Corollary~\ref{cor:equiv} yields:

\begin{cor}\label{cor:syntequiv} Every $L_2$-sentence is $T$-equivalent to some $L_c$-sentence.
\end{cor}

\noindent
The above proof of Corollary~\ref{cor:syntequiv} depends on CH, 
but $T$ has an explicit 
axiomatization and so the statement of this corollary is ``arithmetic''.
Therefore this proof can be converted to one using just ZFC (without CH).
Thus as an obvious  consequence of Corollary~\ref{cor:syntequiv},
 Theorem~\ref{MainTheorem} also holds without assuming CH.


\section{Existence of $\k$, $s$, $c$}

\noindent
In this section we construct under certain conditions a lift $\k$,
a cross-section $s$, and a map $c$ as in the previous section.

\begin{prop}\label{facts} Assume $\mathcal{K} = (K, \Gamma; v, s, c)\models \Mo(\ell,c,s)$.
Then
\begin{align*} s(\ker(c))\ =\ C^\times &\cap s(\Gamma)\quad \text{$($so $\ker(c)\ \subseteq\ v(C^\times))$}, \qquad 
c\big(v(C^\times)\big)\ \subseteq\ \k^\dagger,\\
c(\Gamma)\cap \k^{\dagger}\ &=\ \{0\}\ \Longleftrightarrow\ \ker(c)\ =\ 
v(C^{\times}).
\end{align*}
\end{prop}
\begin{proof} Let $\gamma\in \Gamma$. If $c(\gamma)=0$, then $s(\gamma)^\dagger=0$, so $s(\gamma)\in C^\times\cap s(\Gamma)$. If $s(\gamma)\in C^\times$,
then $c(\gamma)=s(\gamma)^\dagger=0$, so  $\gamma\in  \ker(c)$. This proves the 
first equality.
Next, for the inclusion $c\big(v(C^\times)\big)\subseteq \k^\dagger$, 
suppose $\gamma=va$ with $a\in C^{\times}$. 
Then $s(\gamma) = ua$ with $u\asymp 1$ in $K$, so $u =d(1+\epsilon)$
with $d\in \k^\times$ and $\epsilon\prec 1$. Hence 
$$c(\gamma)\ =\ s(\gamma)^\dagger\ =\ u^\dagger\ =\ d^\dagger + (1 + \epsilon)^\dagger\ =\ d^\dagger+\frac{\epsilon'}{1+\epsilon}.$$
Since $c(\gamma), d^\dagger\in \k$ and $\epsilon'\prec 1$, this gives
$\epsilon'=0$, so $c(\gamma)\in \k^\dagger$, as claimed.
As to the equivalence, suppose $c(\Gamma)\cap \k^{\dagger}= \{0\}$.
Then $c\big(v(C^\times)\big)=\{0\}$ by the inclusion that we just proved, so
$v(C^\times)\subseteq \ker(c)$. We already have the reverse inclusion, so
$\ker(c)= v(C^{\times})$. For the converse, assume $\ker(c)= v(C^{\times})$. Let
$\gamma\in \Gamma$ be such that $c(\gamma)=d^\dagger$ with $d\in \k^\times$. Then
$s(\gamma)^\dagger=d^\dagger$, so $s(\gamma)/d\in C^\times$, hence 
$\gamma=v\big(s(\gamma)/d\big)\in v(C^\times)$, and thus
$c(\gamma)=0$, as claimed.
\end{proof}

\noindent
Examples where $c(\Gamma)\cap \k^{\dagger}\ne \{0\}$: 
Take any differential field $\k$ with $\k\ne C_{\k}$, and take $\Gamma=\Z$. 
Then $\k^\dagger\ne \{0\}$; take any nonzero element $u\in \k^\dagger$.
Then for the additive map $c: \Gamma\to \k$ given by $c(1)=u$ we have
$c(\Gamma)=\Z u\subseteq \k^\dagger$, and so $\k((t^\Gamma))_c$ 
is a model of $\Mo(\ell,c,s)$ with  $c(\Gamma)\cap \k^{\dagger}\ne \{0\}$.
By taking $\k$ to be linearly surjective, this model is $\d$-henselian. 

\medskip\noindent
An example where $c(\Gamma)\cap \k^{\dagger}= \{0\}$: Take
$\k= \T_{\log}$, the differential field of logarithmic transseries; see
\cite[Chapter 15 and Appendix A]{ADAMTT} about $\T_{\log}$, especially the
fact that $\T_{\log}$ is linearly surjective. Also $\T_{\log}$ contains
$\R$ as a subfield, and $f^\dagger \notin \R$ for all nonzero $f\in \T_{\log}$.
Next, take $\Gamma=\R$ and define $c: \Gamma\to \k$ by $c(r)=r$. Then
$K:= \k((t^\Gamma))$ yields a $\d$-henselian model $K_c$ of $\Mo(\ell,c,s)$ 
with $c(\Gamma)\cap \k^{\dagger}= \{0\}$. 
Allen Gehret conjectured an axiomatization of $\Th(\T_{\log})$ that would 
imply its decidability, and thus the decidability of the theory of $K_c$.
This $K_c$ has few 
constants by the following obvious consequence of Proposition~\ref{facts}:

\begin{cor} \label{fewConstants}
Suppose $\mathcal{K} = (K, \Gamma; v, s, c)\models \Mo(\ell,c,s)$. Then:
$$c \text{ is injective and }c(\Gamma)\cap \k^{\dagger}= \{0\}\ \Longleftrightarrow\ \mathcal{K} \text{ has few constants}.$$
\end{cor} 

\medskip\noindent
We now provide an example to show that in Theorem~\ref{F1} we cannot drop 
the map $c$ in the case of few constants. Take $\k = \T_{\log}$ and $\Gamma = \zz$. Define the additive maps $c_1 : \Gamma \to \k$ by $c_1(1) = 1$ and $c_2 :\Gamma \to \k$ by $c_2(1) = \sqrt{2}$; instead of $\sqrt{2}$, any irrational real number will do. Let $K_1 := \k((t^\Gamma))$ and $K_2 := \k((t^\Gamma))$ be the differential Hahn fields with derivations defined as in the introduction using the maps $c_1$ and $c_2$, respectively. They are $\d$-henselian monotone valued differential fields. As in the previous example they have few constants by Corollary~ \ref{fewConstants}. 
We claim that $K_1$ and $K_2$ are not elementarily equivalent as valued 
differential fields (without $c_1$ and $c_2$ as primitives), 
so the traditional Ax-Kochen-Ershov principle 
does not hold.
In $K_1$, we have $t^\dagger = c(1) = 1$ and so $K_1 \models \exists a \neq 0 (a^\dagger = 1)$.
We now show that $K_2 \not\models \exists a \neq 0 (a^\dagger = 1)$.
Towards a contradiction, assume $a \in K_2^\times$ is such that $a^\dagger = 1$.
Then $a = t^k d (1 + \epsilon)$ with $k \in\zz$, $d \in \k^\times$ and $\epsilon \in K_2$ with $\epsilon \prec 1$. Hence $a^\dagger = c_2(k) + d^\dagger + (1 + \epsilon)^\dagger$, so $$k\sqrt{2} + d^\dagger + \frac{\epsilon'}{1 + \epsilon}\ =\ 1.$$
Since $\epsilon' \prec 1$ we get $k\sqrt{2} + d^\dagger = 1$ and $\epsilon' = 0$. Thus $d^\dagger = 1 - k\sqrt{2} \in \mathbb{R}$. Since
$1 - k\sqrt{2} \neq 0$, this contradicts 
$\T_{\log}^\dagger \cap \mathbb{R} = \{ 0 \}$.

Next we give an example of a decidable $\d$-henselian monotone valued differential field with few constants. The valued differential field $\T$
of transseries is linearly surjective by \cite[Corollary 15.0.2]{ADAMTT} and \cite[Corollary 14.2.2]{ADAMTT}.
As $\T[\mathrm{i}]$ with $\mathrm{i}^2 = -1$ is algebraic over $\T$,  it is also linearly surjective by \cite[Corollary 5.4.3]{ADAMTT}. The proof of \cite[Proposition 10.7.10]{ADAMTT} gives $(\T[\mathrm{i}]^\times)^\dagger = \T + \mathrm{i}\der \smallo$, where $\smallo$ is the maximal ideal of the valuation ring of $\T$. Thus taking $\k = \T[\mathrm{i}]$, $\Gamma = \R$ and the additive map $c:\Gamma \to \k$ given by $c(r) = \mathrm{i}r$, we have $c(\Gamma) \cap \k^\dagger = \mathrm{i}\R \cap (\T + \mathrm{i}\der \smallo) = \{ 0 \}$ and therefore $K := \T[\mathrm{i}]((t^\R))_c$ will be a monotone $\d$-henselian valued differential field with few constants by Corollary \ref{fewConstants}. Moreover, $\text{Th}(K)$ is decidable by Theorem~\ref{F1}, since the 2-sorted structure $(\T[\mathrm{i}], \R; c)$ is interpretable in the valued differential field $\T$ and the latter has decidable theory by \cite[Corollary 16.6.3]{ADAMTT}.


\medskip\noindent
\textit{In what follows we fix a differential field $K$ with a valuation
$v: K^\times \to \Gamma=v(K^\times)$ such that $(K, \Gamma;v)$ is a monotone
valued differential field.}

\begin{lem} \label{GroupValuation}
Suppose $(K, \Gamma;v)$ is $\d$-henselian and $\k$ is a lift of its
differential residue field. Then 
$G := \{ a\in K^\times : a^\dagger \in \k \}$ is a subgroup of $K^\times$ 
with $v(G) = \Gamma$.
\end{lem}

\begin{proof} Using $(a/b)^\dagger=a^\dagger-b^\dagger$ for $a,b\in K^\times$ we
see that $G$ is a subgroup of $K^\times$. Let $\gamma \in \Gamma$;
our goal is to find a $g\in G$ with $vg=\gamma$.
Take $f\in K^\times$ with $vf=\gamma$. If $f'\prec f$, then \cite[7.1.10]{ADAMTT}
gives $g\in C^\times$ such that $f\asymp g$, so $g\in G$ and $vg=\gamma$.
Next, suppose $f'\asymp f$. Then $f^\dagger\asymp 1$, so 
$f^\dagger = a + \epsilon$ with $a\in \k$ and $\epsilon \in \smallo$.
By \cite[Corollary 7.1.9]{ADAMTT} we have $\smallo = (1 + \smallo)^\dagger$, so
$\epsilon = (1 + \delta)^\dagger$ with $\delta\in \smallo$. Then
$ (\frac{f}{1+\delta})^\dagger = a\in \k$, so 
$\frac{f}{1 + \delta}\in G$ and $v(\frac{f}{1 + \delta}) =  \gamma$. 
\end{proof}

\noindent
Recall that if $(K, \Gamma;v)$ is $\d$-henselian, then a 
lift of the differential residue field exists. Below we assume a lift 
$\k$ of the differential
residue field is given, and we consider the 
$2$-sorted structure $\big((K,\k), \Gamma;v\big)$ (so $\k$ is a distinguished
subset of $K$).

\begin{lem} \label{SaturatedCaseCrossSectionExistence}
Suppose $\big((K,\k), \Gamma; v\big)$ is $\d$-henselian, $\aleph_1$-saturated and $G$ is a definable subgroup of $K^\times$ such that $v(G) = \Gamma$. Then there exists a cross-section $s:\Gamma \to K^\times$ such that $s(\Gamma)\subseteq G$.
\end{lem}

\begin{proof}
First note that $H := \mathcal{O}^\times \cap G$ is a pure subgroup of $G$.
The inclusion $H\to G$ and the restriction of the valuation 
$v$ to $G$ yield an exact sequence
$$1 \to H \to G \to \Gamma \to 0$$
of abelian groups. Since $H$ is $\aleph_1$-saturated as an abelian group,
this exact sequence splits; see \cite[Corollary 3.3.37]{ADAMTT}. This yields a cross-section $s:\Gamma \to K^\times$ with $s(\Gamma)\subseteq G$.
\end{proof}

\noindent
Combining the previous two lemmas gives us the main result of this section:

\begin{thm}\label{ConstructionOfMaps}
Suppose $\big((K,\k), \Gamma; v\big)$ is $\d$-henselian and 
$\aleph_1$-saturated. Then there is a cross-section $s:\Gamma \to K^\times$  and 
an additive map $c: \Gamma \to \k$ with $s(\gamma)^\dagger = c(\gamma)$ for all $\gamma \in \Gamma$.
\end{thm}

\begin{proof} Since $\k$ is now part of the structure, the subgroup $G$ of $K^\times$ from Lemma~\ref{GroupValuation} is definable. Now apply Lemma \ref{SaturatedCaseCrossSectionExistence} and get a cross-section $s:\Gamma \to K^\times$ such that $s(\Gamma)^\dagger \subseteq \k$. Take the additive map $c:\Gamma \to \k$ to be given by $c(\gamma) = s(\gamma)^\dagger$.
\end{proof}

\begin{proof}[Proof of Theorem \ref{F2}] Let a monotone $\d$-henselian 
valued field be given. Then it has a lift of its differential residue field, 
and fixing such a lift $\k$, it is a structure 
$\big((K,\k), \Gamma; v\big)$ as above. Passing to an elementary extension, 
we can assume $\big((K,\k), \Gamma; v\big)$ is $\aleph_1$-saturated.
Then Theorem \ref{ConstructionOfMaps} yields a cross-section $s: \Gamma\to K^\times$ and 
an additive map $c: \Gamma \to \k$ with $s(\gamma)^\dagger = c(\gamma)$ for all $\gamma \in \Gamma$.
This in turn yields a Hahn field
$\k((t^\Gamma))_c$ that is elementarily equivalent to 
$\big((K,\k), \Gamma; v,s,c\big)$. 
\end{proof}


\noindent
We can now prove Theorem \ref{F4}:

\begin{proof}[Proof of Theorem \ref{F4}]
Let $F$ be a monotone $\d$-henselian valued field such that $v_F(C_F^\times)$ is pure in $\Gamma_F=v_F(F^\times)$. 
The valued differential field $F$ has a lift of its differential residue field, 
and fixing such a lift $\k_F$ we get the structure $\big((F,\k_F), \Gamma_F; v_F\big)$.
Take an elementary extension $\big((K,\k), \Gamma; v\big)$ of it that is $\aleph_1$-saturated. Then $\Delta := v(C_K^\times)$ is pure in $v(K^\times)$.
Since $\Delta$ is also $\aleph_1$-saturated (as an abelian group), 
we have a direct sum decomposition $\Gamma = \Delta \oplus \Gamma^*$ by \cite[Corollary 3.3.37]{ADAMTT}.
Since the valued subfield $C := C_K$ of $K$ is $\aleph_1$-saturated, it has
a cross-section $s_C : \Delta \to C^\times$.
Theorem~\ref{ConstructionOfMaps} yields a cross-section $\tilde{s}: \Gamma\to K^\times$ of
the valued field $K$ such that $\tilde{s}(\Gamma)^\dagger \subseteq \k$.
By the definition of $\Delta$ we have $\tilde{s}(\gamma)\notin C$ for all
$\gamma\in \Gamma\setminus \Delta$. 

Let $s$ be the cross-section of
the valued field $K$ that agrees with $s_C$ on $\Delta$ and with $\tilde{s}$ on $\Gamma^*$. Then $s(\gamma)^\dagger \in \k$ for all $\gamma\in \Gamma$, so we have an additive map $c:\Gamma\to \k$ given by 
$c(\gamma)=s(\gamma)^\dagger$. Moreover, for $\gamma\in \Gamma$,
$$ c(\gamma) = 0\ \Leftrightarrow\ s(\gamma)'=0\ \Leftrightarrow\ 
s(\gamma)\ \in C\ \Leftrightarrow\ \gamma \in \Delta.$$
This gives $\ker(c) = v(C^\times)$, and thus 
$c(\Gamma) \cap \k^\dagger = \{ 0 \}$ by Proposition~\ref{facts}. Since $\ker(c)$ is a pure subgroup of $\Gamma$ then so is $\Delta$.
This in turn yields a Hahn field $\k((t^\Gamma))_c$ with the required properties that is elementarily equivalent to $\big((K,\k), \Gamma; v,s,c\big)$. 
\end{proof}

\section{Eliminating the cross-section}

\medskip\noindent
Note that every $\mathcal{K}\models \Mo(\ell,s,c)$ satisfies the sentences
\begin{enumerate}
\item $\forall \gamma \forall \delta \quad c(\gamma + \delta) = c(\gamma) + c(\delta)$,
\item $\forall \gamma \exists x\ne 0 \quad v(x) = \gamma\ \&\ x^\dagger = c(\gamma)$.
\end{enumerate}
These sentences don't mention the cross-section $s$. Below we 
derive the analogue of Theorem~\ref{MainTheorem} in the setting without a cross-section. Let $L_2^{-}$ be the language $L_2$ with the symbol $s$ for the cross-section removed. Let $\Mo(\ell, c)$ be the $L_2^{-}$-theory whose models are the
$L_2^{-}$-structures
$$\mathcal{K}\ =\ (K, \Gamma; v, c),$$
where $K$ is a differential field equipped with a differential
subfield $\k$ (singled out by a unary 
predicate symbol), $\Gamma$ is an ordered abelian group,
$v: K^\times \to \Gamma=v(K^\times)$ is a valuation that makes $K$ into
a monotone valued differential field such that $\k \subseteq K$ is a lift
of the differential residue field, and $c: \Gamma \to \k$
is such that the sentences (1) and (2) above are satisfied. 

\begin{lem} \label{ExistenceOfCrossSectionWithMapCOnly}
Suppose $\mathcal{K}=(K, \Gamma; v, c)\models \Mo(\ell,c)$ is $\d$-henselian and $\aleph_1$-saturated. Then there is a cross-section $s:\Gamma \to K^\times$ such that $s(\gamma)^\dagger = c(\gamma)$ for all $\gamma\in \Gamma$.
\end{lem}

\begin{proof}
By (1) and (2) we have a definable subgroup  
$G := \{ x \in K^\times : x^\dagger = c( v(x) )\}$ of $K^\times$
with $v(G) = \Gamma$. 
Now, use Lemma \ref{SaturatedCaseCrossSectionExistence} to get a cross section $s:\Gamma \to K^\times$ with $s(\Gamma)\subseteq G$. This $s$ has the desired property.
\end{proof}

\begin{thm}\label{MT-s} 
Suppose $\mathcal{K}=(K,\Gamma;v,c)\models \Mo(\ell,c)$ is $\d$-henselian. Then
$\Th(\mathcal{K})$ is axiomatized by the following axiom schemes:
\begin{enumerate}[font=\normalfont]
\item $\Mo(\ell, c)$;
\item the axioms for $\d$-henselianity;
\item $\operatorname{Th}(\k,\Gamma;c)$ with $\k$ as differential field and $\Gamma$ as ordered abelian group.
\end{enumerate}
\end{thm}
\begin{proof}
Let any two $\aleph_1$-saturated models of the axioms in the theorem be given. 
By Lemma~\ref{ExistenceOfCrossSectionWithMapCOnly} we have in both models a cross-section that make these into models of $\Mo(\ell, s,c)$. It remains to appeal to Theorem \ref{MainTheorem}
to conclude that these two models are elementarily equivalent.
\end{proof}

\noindent
Before giving the proof of Corollary \ref{F3} from the introduction we 
note that any algebraic valued 
differential field 
extension of a monotone valued differential field is again monotone; see \cite[Corollary 6.3.10]{ADAMTT}. 

\begin{proof}[Proof of Corollary \ref{F3}] Let $K$ range over $\d$-henselian
monotone valued differential fields.
As in \cite[Proof of Corollary 8.0.2]{ADAMTT} we have a set $\Sigma_n$ of sentences in the language of valued differential fields, independent of $K$, such that $K \models \Sigma_n$ if and only if every valued differential field extension $L$ of $K$ with $[L:K] = n$ is $\d$-henselian.
Now by Theorem \ref{F2} we have $K\equiv \k((t^\Gamma))_c$ for a suitable
differential field $\k$, ordered abelian group $\Gamma$, and additive map
$c: \Gamma \to \k$. Every valued differential field extension $L$ of $\k((t^\Gamma))_c$ of finite degree is spherically complete as a valued field and so $\d$-henselian by \cite[Corollary 5.4.3 and Theorem 7.2.6]{ADAMTT}. Hence
$\k((t^\Gamma))_c\models \Sigma_n$ and thus $K \models \Sigma_n$, for all $n\ge 1$.
\end{proof}

\noindent
We now give an example of a monotone $\d$-henselian field $F$  such that $v(C_F^\times)$ is not pure in $v(F^\times)$. This elaborates on an example by the
referee
of a monotone henselian valued differential field $F$ for which 
$v(C_F^\times)$ is not pure in $v(F^\times)$.

Let the additive map $c:\zz \to \T_{\log}$ be given by $c(1) = 1$. With the usual
derivation on $\T_{\log}$, this yields the (discretely) 
valued differential field $\k = \T_{\log}((s^\zz))_c$, with $s'=~{}s$. Since $\T_{\log}$ is
linearly surjective, $\k$ is $\d$-henselian field and thus linearly surjective.
We now forget about the valuation of $\k$, consider it just as a differential
field, and introduce $K := \k((t^\zz))_d$ with the additive map $d:\zz \to \k$ given by $d(1) = 0$, so $t'=0$. 
Then $K$ is a monotone $\d$-henselian field with $v(K^\times) = \zz$.
Finally, let $F := K(\sqrt{st})$, which is naturally a valued differential 
field extension of $K$.
Since $F$ is algebraic over $K$, it is monotone and $\d$-henselian too, by Corollary~\ref{F3}.  Clearly, $v(F^\times) = \frac{1}{2}\zz$.
We claim that $v(C^\times_F) = \zz$ and so it is not pure in $v(F^\times)$.
From $t^{\zz}\subseteq C_F$ we get $\zz \subseteq v(C_F^\times)$. For the reverse inclusion, let any element $a + b\sqrt{st}\in C_F^\times$ be given
with $a, b \in K$, not both zero. 
Now, $$(a+b\sqrt{st})' = a' + b'\sqrt{st} + b (\sqrt{st})' = a' + b' \sqrt{st} + b(\sqrt{st}/2) = a' + (b' + b/2)\sqrt{st},$$
so $a' = 0$ and $b' + b / 2 = 0$.
From $b' = -b / 2$ we now derive $b = 0$. (Then $a + b\sqrt{st} = a\in C_{\k}((t^\zz))$, and thus $v(a + b\sqrt{st})\in \zz$, as claimed.)
Let $k,l$ range over $\zz$. Towards a contradiction, suppose $\displaystyle b = \sum_{l \geq l_0} b_l t^l$ with all $b_l \in \k$, $l_0 \in \zz$, $b_{l_0}\ne 0$.
Then $\displaystyle b' = \sum_{l \geq l_0} b'_l t^l$ and so the equality $b'=-b/2$ takes the form
$$\sum_{l \geq l_0} b'_l t^l\ =\  -\frac{1}{2} \sum_{l \geq l_0} b_l t^l\ =\ \sum_{l \geq l_0} -\frac{1}{2} b_l t^l.$$
Therefore $b'_l = - b_l/2$ for all $l\geq l_0$, in particular for $l = l_0$. 
Assume $\displaystyle b_{l_0} = \sum_{k\geq k_0} u_k s^k$, with all $u_k \in \T_{\log}$, and $k_0\in \zz$, $u_{k_0} \neq 0$.
We have 
$\displaystyle b'_{l_0} = \sum_{k\geq k_0} (u'_k + k u_k) s^k$
and
$\displaystyle -\frac{1}{2} b_{l_0} = \sum_{k\geq k_0} -\frac{1}{2}u_k s^k$.
Thus $u'_k + ku_k = -u_k/2$ for all $k \geq k_0$. For $k = k_0$ we have $u_{k_0} \neq 0$, and so this gives $u^\dagger_{k_0} = -k_0 - 1/2$. 
However, this contradicts $\T_{\log}^\dagger \cap \mathbb{R} = \{ 0 \}$ and hence the claim is proved.

On the other hand:

\begin{prop} Let $F$ be a henselian valued differential
field with algebraically closed or real closed
residue field.  Then $v(C_F^\times)$ is pure in $v(F^\times)$. 
\end{prop}
\begin{proof} Let $n\alpha = \beta$ with 
$\alpha \in v(F^\times), \beta \in v(C_F^\times), n\ge 1$; 
our job is to show that then $\alpha \in v(C_F^\times)$.
Take $a \in F^\times$ with $v(a) = \alpha$ and $b \in C_F^\times$ with $v(b) = \beta$, so $v(b/a^n)=0$; if the residue field is real closed we also arrange that the residue class of $b/a^n$ is positive.
Considering the polynomial $\displaystyle P(Y) = Y^n - (b/a^n)\in \mathcal{O}_F[Y]$, the henselianity of $F$ and the assumption on the residue field gives
a zero $y \asymp 1$ in $F$ of $P$. Then $(ay)^n = b \in C_F^\times$, hence $ay \in C_F^\times$ with $v(ay) = \alpha$.
\end{proof}

\medskip\noindent
A valued differential field with small derivation is said to be 
{\em $\d$-algebraically maximal\/} if it has no proper immediate 
$\d$-algebraic valued differential field extension. 
For monotone valued differential fields with
linearly surjective differential residue field, 
$$\text{$\d$-algebraically maximal}\ \Longrightarrow\ \text{$\d$-henselian}$$
by \cite[Theorem 7.0.1]{ADAMTT}.  By \cite[Theorem 7.0.3]{ADAMTT}, 
the converse holds
in the case of few constants, but an example at the end of Section 7.4 of \cite{ADAMTT} shows that this converse fails for some $\d$-henselian monotone valued differential field with many constants. Below we generalize this example as follows:

\begin{cor}
Let $K$ be a $\d$-henselian, monotone, valued differential field with $v(C^\times)\ne \{0\}$. Then some $L\equiv K$ is not $\d$-algebraically maximal.
\end{cor}

\begin{proof}
By Theorems ~\ref{F1} and ~\ref{F2} and L\"owenheim-Skolem we can arrange $K= \k((t^\Gamma))_c$ where the differential field $\k$ and the ordered abelian group $\Gamma$ are countable and $c: \Gamma \to \k$ is additive.
With $C:= C_K$, take $a\in C^\times$ with $va=\gamma_0>0$.
Then $a = \sum_{\gamma\ge \gamma_0} a_\gamma t^\gamma$, with 
 $\der(a_\gamma) + c(\gamma) a_\gamma = 0$ for all $\gamma$, in particular
for $\gamma=\gamma_0$. Hence $\mathfrak{m} := a_{\gamma_0} t^{\gamma_0}\in C$, and so all infinite sums $\sum_n q_n \mathfrak{m}^n$ with rational $q_n$ lie in $C$ as well. Thus $C$ is uncountable. 

On the other hand, $\k(t^\Gamma)$ is countable and so by L\"owenheim-Skolem
we have a countable $L\prec K$ that contains $\k(t^\Gamma)$. 
Thus $K$ is an immediate extension of $L$ and we can take $a\in C\setminus L$.
Then $L\< a \> = L(a)$ is a proper immediate $\d$-algebraic extension of $L$ and therefore $L$ is not $\d$-algebraically maximal.
\end{proof}

\section{Eliminating the lift of the differential residue field}

\medskip\noindent
In this section we drop the requirement of having a {\em lift} of the differential residue field in our structure and instead use a copy of the differential residue field.
For this purpose we consider 3-sorted structures
$$\mathcal{K} = (K, \mathbf{k}, \Gamma; \pi, v, c)$$
where $K$ and $\mathbf{k}$ are differential fields, $\Gamma$ is an ordered abelian group,
$v:K^\times \to \Gamma$ is a valuation which makes $K$ into a monotone valued differential field,
$\pi:\mathcal{O} \to \mathbf{k}$ with $\mathcal{O} := \mathcal{O}_v$ is a surjective differential ring morphism,
$c:\Gamma \to \mathbf{k}$ is an additive map satisfying
$\forall \gamma \exists x\ne 0  \quad \Big[v(x) = \gamma \ \&\  \pi(x^\dagger) = c(\gamma)\Big]$.
We construe these $\mathcal{K}$ as $L_3$-structures for a natural $3$-sorted language $L_3$ (with unary function symbols for $\pi, v$ and $c$). We have an obvious set $\textnormal{Mo}(c)$ of $L_3$-sentences whose models are exactly these $\mathcal{K}$.

\begin{lem}\label{ExistenceOfLiftWithMapCOnly}
Suppose $\mathcal{K} = (K, \mathbf{k}, \Gamma; \pi, v, c) \models \textnormal{Mo}(c)$ is $\d$-henselian and $\k$ is any lift of the differential residue field. Then $(K, \Gamma; v, \iota \circ c) \models \textnormal{Mo}(l, c)$, where  $\iota : \mathbf{k} \to \k$ is the inverse of the differential field isomorphism $\left.\pi\right|_\k :\k \to \mathbf{k}$.
\end{lem}

\begin{proof}
We need to check the two conditions from the previous section.
First of all $\iota \circ c$ is obviously additive.
Fix $\gamma \in \Gamma$.
There is an element $x \in K^\times$ with $v(x) = \gamma$ and $\pi(x^\dagger) = c(\gamma)$.
Let $a = (\iota \circ \pi) (x^\dagger) = (\iota \circ c) (\gamma)$. As $a \in \k$ and $\pi(a) = \pi(x^\dagger)$, we get $x^\dagger = a  + \epsilon$ for some $\epsilon \prec 1$. By \cite[Corollary 7.1.9]{ADAMTT} we have $ \epsilon = (1 + \delta) ^ \dagger$ for some $\delta \prec 1$ and thus 
$$(\iota \circ c) (\gamma) = a = x^\dagger - (1 + \delta)^\dagger = \Big( \frac{x}{1 + \delta} \Big)^\dagger, \textnormal{ and } v \Big( \frac{x}{1 + \delta} \Big) = v(x) = \gamma. $$
This completes the proof of the lemma.
\end{proof}

\begin{thm}\label{MT-l} 
Suppose $\mathcal{K}=(K, \mathbf{k}, \Gamma; \pi, v, c)\models \Mo(c)$ is $\d$-henselian. Then
$\Th(\mathcal{K})$ is axiomatized by the following axiom schemes:
\begin{enumerate}[font=\normalfont]
\item $\Mo(c)$;
\item the axioms for $\d$-henselianity;
\item $\operatorname{Th}(\mathbf{k},\Gamma;c)$ with $\mathbf{k}$ as differential field and $\Gamma$ as ordered abelian group.
\end{enumerate}
\end{thm}
\begin{proof}

Let any two $\aleph_1$-saturated models $\mathcal{K}_1=(K_1, \mathbf{k}_1, \Gamma_1; \pi_1, v_1, c_1)$ and $\mathcal{K}_2 = (K_2, \mathbf{k}_2, \Gamma_2; \pi_2, v_2, c_2)$ of the axioms in the theorem be given.
By Lemma~\ref{ExistenceOfLiftWithMapCOnly} we have in both models lifts of the differential residue fields that make these into models of $\Mo(\ell, c)$. So $\Th(\mathbf{k}_i, \Gamma_i; c_i) = \Th(\k_i, \Gamma_i; \iota_i \circ c_i)$ where $\iota_i$ is the isomorphism between the differential residue field $\mathbf{k}_i$ and its lift $\k_i$ for $i=1,2$.
It remains to appeal to Theorem \ref{MT-s}
to conclude that these two models are elementarily equivalent.
\end{proof}

\section*{Acknowledgments}
\noindent
The author thanks Lou van den Dries for numerous discussions and comments on this paper. The author also thanks the referee for pointing out an error in the previous version of the paper and for making other helpful comments.

\end{document}